\documentclass[11pt]{article}
\usepackage[authoryear]{natbib}

\usepackage[margin=1in,footskip=0.25in]{geometry}

\usepackage{xcolor}
\usepackage{adjustbox}
\usepackage{comment}
\usepackage{xcolor}
\usepackage{setspace}
\usepackage{placeins,paralist}
\usepackage[inline]{enumitem}
\usepackage{amsthm,amsfonts}
\usepackage{mathtools, nccmath}
\usepackage{natbib}
\usepackage{graphicx, booktabs, multirow}
\usepackage[ruled,vlined]{algorithm2e}
\usepackage[hypcap=false]{caption}
\usepackage[toc,page,title]{appendix}
\usepackage{mathrsfs}
\usepackage{url}
\usepackage{hyperref}

\newtheorem{property}{Property}

\newtheorem{theorem}{Theorem}[section]
\newtheorem{lemma}[theorem]{Lemma}
\theoremstyle{definition}

\newlist{choices}{enumerate*}{1}
\setlist[choices]{itemsep = 1.125in, label=(\Alph*)}

\usepackage{authblk}

\title{\Large\bf The Stochastic Team Orienteering Problem}
\author[1]{Alberto Guastalla}
\author[2]{Jean-François Côté}
\author[1]{Roberto Aringhieri}

\affil[1]{\footnotesize Department of Computer Science, University of Turin, Italy}

\affil[2]{\footnotesize CIRRELT, Université Laval, Québec, Canada}

\begin{document}

\maketitle

\begin{abstract}
This paper analyses the Stochastic Team Orienteering Problem (STOP), a stochastic variant of the Team Orienteering Problem (TOP). In the STOP, travel times are represented by random variables. The objective is to determine a set of routes that maximises the expected collected profit.
We model the problem as a two-stage stochastic integer program and propose an exact solution method based on the Integer L-shaped method. We also consider a set of non-linear chance constraints to restrict the search to highly reliable routes.
In the first-stage, a set of routes is selected, while in the second-stage the recourse evaluates the expected profit of the selected routes. We present a new set of optimality cuts and strengthen the formulation by considering valid inequalities and constraints liftings. Computational results provide a comprehensive analysis of the algorithm using the standard TOP benchmark dataset from the literature.
\end{abstract}


\section{Introduction}
\label{sec:introduction}

The Team Orienteering Problem (TOP) is a well-known combinatorial optimisation problem defined on a complete directed graph in which each node has an associated profit and each arc has an associated travel time. The objective is to find a fixed set of routes that collects the maximum total profit from the visited nodes without exceeding a deterministic time budget on any route. Each route must start at the source node and end at the destination node.

Travel times are often subject to uncertainty caused by factors such as weather conditions, traffic congestion, or other unforeseen events. As a result, a solution derived from the TOP model may become infeasible or suboptimal when implemented. In this paper, we address the challenge of constructing a TOP solution under the assumption that the probability distribution of the uncertain travel times is known beforehand. It is worth noting that, although we are dealing with continuous probability distributions, we do not employ scenario-generation procedures to estimate the value of an optimal solution. The purpose of this work is to introduce the Stochastic Team Orienteering Problem (STOP) as a two-stage stochastic integer program with chance constraints that explicitly account for route reliability. The first-stage is adapted from the flow-based formulation reported in~\citet{Guastalla2025}, while the second-stage (the recourse problem) evaluates the expected profit associated with a specific first-stage solution. The objective of the STOP is to maximise the expected profit collected from the selected routes.
We summarise our contributions as follows:
\begin{inparaenum}[(i)]
\item we present the STOP as a two-stage stochastic integer linear problem;
\item we propose an exact algorithm based on the Integer L-shaped method. To the best of our knowledge, this is the first study to employ the L-shaped method in a routing problem with a maximisation objective;
\item we introduce a new set of optimality cuts designed to guarantee the algorithm's convergence;
\item we present a new family of valid inequalities for maximisation problems that we call Upper-Bounding Functionals (UBFs);
\item we introduce a set of valid inequalities to strengthen the objective function value and propose two constraints liftings to further tighten the formulation of the first-stage problem.
\end{inparaenum}


The paper is organised as follows.
Section~\ref{sec:literature} provides a comprehensive review of the literature, with a focus on the most relevant stochastic applications of Orienteering Problems (OPs) and the Integer L-shaped method.
Section~\ref{sec:problem} presents the problem definition for the STOP along with the mathematical formulation.
Section~\ref{sec:algorithm2} presents the exact algorithm for solving the STOP. 
Finally, Section~\ref{sec:analysis} presents the computational results used to evaluate the algorithm’s performance on the TOP benchmark dataset and to provide additional comparisons.
%
%
Finally, Section~\ref{sec:end} concludes the paper.

\section{Literature review}
\label{sec:literature}

This section provides a comprehensive review of the most relevant and recent studies related to the STOP. It is organised into two parts. The first part examines stochastic applications of Orienteering Problems (OPs), while the second part focuses on the Integer L-shaped solution method.



\subsection{Applications}

\citet{Campbell2011} introduced the Orienteering Problem with Stochastic Travel and Service Times, a variant of the orienteering problem in which travel and service times follow specified probability distributions. In this setting, a reward is earned when a commitment to a customer is fulfilled by the end of the day, whereas failing to fulfil the commitment results in a penalty. This model captures the operational challenge faced by companies that may have more customers than they can serve in a single day. The authors identified specific instances that can be solved exactly and showed how existing variable neighbourhood search heuristics can be applied to more general cases.

\citet{Evers2014} investigated the Orienteering Problem with Stochastic Weights to account for uncertainty arising in applications such as logistics and tourism. The authors formulated a two-stage stochastic integer program with the objective of maximising the expected collected profit. To address the nonlinearity of the expected profit, they introduced a linearisation method to estimate the total profit achievable for a given route across different weight-realisation scenarios, thereby enabling the application of the sample average approximation method. Although the solution algorithm asymptotically converges to the optimal solution of the two-stage model, it can be computationally expensive. To handle larger instances, they developed a heuristic approach that exploits the problem structure and explicitly addresses the associated uncertainty.


More recently, \citet{Angelelli2017} introduced the Probabilistic Orienteering Problem (POP), a variant of the OP defined on a directed graph where each arc has an associated cost, and each node offers a prize which is only available with a certain probability. A vehicle departs from a fixed source with a limited budget to visit a subset of customers and must return to a designated destination. In the first-stage, a subset of customers is chosen and a corresponding path is planned, ensuring that the vehicle can visit all selected customers and still reach the destination within the given budget. Once the availability of each customer in the subset is revealed, the vehicle follows the initial path, omitting the customers that are unavailable. The objective of the problem is to determine a first-stage solution that maximises the expected profit of the second-stage path, defined as the difference between the expected total prize and the expected total cost. 



In contrast, \citet{Thayer2021} introduced a variant of the OP incorporating chance constraints. In this paper, the total travel time of a route is modelled as a random variable, and chance constraints are incorporated to bound the probability that this total travel time exceeds a predefined time budget. The authors formulated this problem as a constrained Markov decision process and applied a Lagrangian approach to solve it.
%
%
Experimental evaluation shows that their method significantly outperforms previous solutions based on linear programming techniques.

\citet{MONTEMANNI2025106947} further investigated the POP (initially introduced by \citet{Angelelli2017}) proposing an iterative model-based algorithm that solves a sequence of deterministic subproblems. The algorithm is capable of identifying and certifying optimal solutions within a given time. Computational experiments show that the proposed method performs competitively when compared with both exact and heuristic algorithms from the existing literature.

\citet{PISINGER2026107365} introduced a two-stage stochastic variant of the TOP in which customers are classified in two categories. The former, are subscription customers that, if selected, must be served in every scenario. The latter are stochastic on-demand customers, which are optional in each scenario. The objective is to select a subset of subscription customers is such a way to maximise the expected total profit from both subscription and on-demand customers. To address this problem, the author proposed two versions of a stochastic metaheursitic approach based on large neighborhood search methodology. The first one, the distributed approach, a high-level local search explores the first-stage decisions, while multiple lower-level local searches independently optimise each scenario. The second one, the integrated approach, employs a single local search procedure which simultaneously handles both first and second-stage decisions.
Computational experiments show that the distributed approach consistently produces better solutions in less time.

Finally, to the best of our knowledge, the only two studies addressing the STOP are \citet{Panadero2017} and \citet{Panadero2023}.
In both articles, the authors proposed a simheuristic approach that integrates biased-randomised heuristics with a variable neighbourhood search framework and Monte Carlo simulation. In contrast, the present paper introduces an exact solution method.


\subsection{Integer L-shaped method}
\label{ssec:literature:ILM}

The Integer L-shaped Method (ILM) was introduced by ~\citet{LAPORTE1993133} as an extension of Benders decomposition to solve binary two-stage stochastic programs by decomposing them into a Master Problem (MP) and a Sub-Problem (SP). The key idea is to compute the objective function of the MP with linear inequalities, constructed iteratively using Benders cuts (feasibility or optimality cuts). Initially, the MP represents a relaxed representation of the first-stage problem. After solving the MP to obtain an integer solution, such a solution is passed to the SP which evaluates its feasibility and computes the recourse cost. If the SP is infeasible, a feasibility cut is generated to eliminate such a solution from the MP. If feasible, the SP generates an optimality cut that tightens the true objective of the MP. Then, the MP is solved again. This iterative process continues until convergence is achieved. Within the stochastic vehicle routing problems literature, we identify three families of optimality cuts, hereafter referred to as: (i) solution-based (or classical) optimality cuts, (ii) route-based optimality cuts, and (iii) customer-based optimality cuts. 



For solution-based optimality cuts, the recourse function is replaced by an auxiliary variable that is bounded using optimality cuts. Each cut is designed to be active for a single first-stage solution. Its first use in a vehicle routing problem is due to \citet{Gendreau1995} who studied a stochastic vehicle routing problem in which each customer is uncertain to be present and has an uncertain demand. Later developments to improve the efficiency are due to the introduction of Lower-Bounding Functionals (LBFs) which bounds the auxiliary variable for more than a single solution. \citet{Hjorring1999} proposed a LBF, named \textit{partial route inequalities}, that are based on the notion of \textit{partial route} which represents an ordered sequence of unordered groups of nodes separated by single nodes. These were extended by \cite{Laporte2002} to the multiple vehicles case, generalised by \cite{Jabali2014}, \cite{salavati2019exact} and \cite{Hoogendoorn2023}. The important drawback of these optimality cuts is that we are required to add a considerable amount of them to obtain convergence.

The route-based optimality cuts extends the classical ones by being defined for each route of a solution. The recourse problem is replaced by a sum of auxiliary variables, one such variable for each customer, and each variable represents the recourse of a route which starts exactly by serving such a customer. These optimality cuts were proposed, but not tested, by \citet{Seguin1994} for a stochastic vehicle routing problem. It was later tested by \citet{Cote2020} on a different stochastic vehicle routing problem, which demonstrated their usefulness for solving more problems efficiently. The main advantage is that a cut is active for all solutions that contain a specific route. This improve by at least one order of magnitude the computation times. More recently, \citet{Hoogendoorn2023} extended the \textit{partial route inequalities} by introducing the \textit{partial route-split} and \textit{multi-route-split inequalities} which also helps the convergence of the proposed ILM.

The customer-based optimality cuts are designed to bound the recourse of paths, which do not require to be connected to the depot. While a route-based optimality cut is only active for solutions that contain the same route, a customer-based optimality cut is active for all solutions that contain such a path. They have the advantage of being active for a broader range of solutions. The recourse function is replaced as well by a sum of auxiliary variables, each one associated to a customer. In this case, each variable can be interpreted as the customer's contribution to the recourse.
%
%
The customer-based optimality cuts were introduced by \citet{Lucas2024} as part of the Disaggregated ILM (DILM) for solving two-stage stochastic integer programs in which first-stage solutions can be decomposed into disjoint components, each one having a monotonic recourse function. In minimisation problems, the monotonicity property requires that the recourse cost of a component is always greater than or equal to the recourse of any of its subcomponents. Exploiting this structure, the authors proposed these optimality cuts, and a new type of LBFs, which are valid if the recourse function is monotonic. Computational results on benchmark instances demonstrated that the method achieves state-of-the-art performance.
Following this line, \citet{Legault2025} corrected \citet{Lucas2024} and showed that a formulation based on the DILM is valid only if the recourse function is superadditive. This requires that the recourse cost of a component must be greater than or equal to the sum of recourse costs of two disjoint subcomponents, in a minimisation problem. Their approach integrates a new set of LBFs that extends the ones introduced in~\citet{Lucas2024}, yielding consistent computational improvements. Numerical experiments confirmed substantial performance gains over the best-known ILMs, certifying the optimality of several open instances.


\section{Problem definition and mathematical formulation}
\label{sec:problem}

We formulate the STOP as a two-stage stochastic integer program. In the first-stage, the decisions involve selecting which nodes to visit and how to assign them to different routes in a specific sequence. The second-stage problem evaluates the solution provided by the first-stage, represented as a set of $m$ routes, by computing its corresponding expected profit.
Finally, we assume an infinite scenario space, as the uncertainty parameters $\xi_{ij}$, representing the travel times between nodes $i$ and $j$, are modelled as continuous independent random variables.


The STOP is formulated on a directed and complete graph $G=(N,A)$, where $N$ is the set of nodes denoted from $1$ (the source or depot) to $n$ (the destination), and $A$ is the set of directed arcs across the nodes in $N$ ($|N|=n$). A path $w$ in a graph $G$ is a sequence of distinct nodes represented by a set of arcs. If $w$ includes neither the source nor the destination node, it is referred to as a customers path. A sub-path of $w$ is a contiguous subsequence of nodes in $w$, eventually $w$ itself. A route is a path where the first node is $1$ and the last node is $n$. The route $r_h$ is said to perform a customers path $h$ if it visits all nodes of $h$ consecutively.

Each arc $(i,j) \in A$ is associated with a Normal independent random variable $\xi_{ij} \sim \mathcal{N}(\mu_{ij}, \sigma_{ij}^2)$ to represent its uncertain travel time. We assume that the triangle inequality holds for both expectations $\mu_{ij}$ and variances $\sigma_{ij}^2$ which ensures that all the sub-paths of a feasible path are feasible.
Set $\hat{N}$ denotes the set of customers and is obtained from $N$ by removing the source and destination nodes. Similarly, $\hat{A}$ is obtained from $A$ by removing all arcs incident to those nodes. Sets $N(w)$ and $A(w)$ denote the set of nodes and arcs that compose path $w$. Sets $\hat{N}(w)$ and $\hat{A}(w)$ correspond to the sets $N(w)$ and $A(w)$, without considering the source and destination nodes for $\hat{N}(w)$ and the arcs connected to them for $\hat{A}(w)$. Each customer $k \in \hat{N}$ is associated with a deterministic non-negative profit $p_k$.

The goal of the STOP is to find a set of $m$ routes that maximise the total expected profit collected while respecting a set of reliability constraints implemented as chance constraints in order to take into account a reliability measure for each route as in~\citet{Thayer2021}.

In the following, we state some useful definitions. The \emph{route uncertain travel time} $\xi_r = \sum_{(i,j) \in A(r)} \xi_{ij}$ is defined as the sum of the random variables corresponding to the arcs composing the route $r$. Accordingly to the closure property of the Normal distribution with respect to the sum, the random variable $\xi_r$ is still distributed as a Normal random variable, i.e.,  $\xi_r \sim \mathcal{N}(\mu_r, \sigma^2_r)$, with the \emph{route expectation} $\mu_r =\sum_{(i,j) \in A(r)} \mu_{ij}$, and \emph{route variance} $\sigma^2_r = \sum_{(i,j) \in A(r)} \sigma^2_{ij}$.
%
%
%
The \emph{route reliability} $\delta_r = \mathbb{P}(\xi_r \leq T_{\max})$ is defined as the probability that $\xi_r$ is less than or equal to the deterministic time budget $T_{\max}$. A route $r$ is said to be feasible if $\delta_r \geq \alpha$ where $\alpha$ is a probability threshold.
The recourse $\mathcal{Q}(r)$ of route $r$ is the product of its reliability and its total collected profit collected:
\begin{equation}
\label{route_recourse}
\mathcal{Q}(r) = \delta_r \sum_{k \in \hat{N}(r)} p_k.
\end{equation}
%

%
%
We propose a compact two-index formulation by leveraging the flow-based formulation presented in~\cite{Guastalla2025} that makes use of the following decision variables:
\begin{inparadesc}
\item$y_{k}$ is equal to $1$ if and only if the node $k \in \hat{N}$ is visited, $0$ otherwise;
\item$x_{ij}$ is equal to $1$ if and only if the arc $(i,j) \in A$ is traversed, $0$ otherwise ($x_{1n}$ is an integer variable representing the number of routes that do not serve any customer);
\item $z_{ij}$ represents the expected arrival time at node $j$ coming from node $i$ and can be thought as the amount of flow that passes through the arc $(i,j) \in A$.
\end{inparadesc}
The mathematical formulation is as follows.


\allowdisplaybreaks{
	\begin{subequations}
		\begin{align}
			\max \quad & {\mathcal{Q}(x)}\label{eq:2STOPof}\\
			\textrm{s.t.}
			\quad & \sum_{(1,j) \in A} x_{1j} = \sum_{(i,n) \in A} x_{in} = m, \label{eq:2STOPstartend}\\
			\quad & \sum_{(i,k) \in A} x_{ik} = \sum_{(k,j) \in A} x_{kj} = y_k, \quad \, k \in \hat{N}, \label{eq:2STOPconnectivity}\\
			\quad & \sum_{(k,j) \in A} z_{kj} - \sum_{(i,k) \in A} z_{ik} = \sum_{(k,j) \in A} \mu_{kj} x_{kj}, \quad \, k \in \hat{N}, \label{eq:2STOPflow}\\
  \quad & \mathbb{P} \left(\sum_{(i,j) \in A(r)} \xi_{ij} x_{ij} \leq T_{\max} \right) \geq \alpha, \quad r \in \mathcal{R}(x),\label{eq:2chanceconstraints}\\
			\quad & z_{ij} \leq (T_{\max} - \mu_{jn}) \, x_{ij}, \quad \, (i,j) \in A, \label{eq:2STOPUB}\\
			\quad & z_{ij} \geq (\mu_{1i} + \mu_{ij}) \, x_{ij}, \quad \, (i,j) \in A, \label{eq:2STOPLB}\\
			\quad & z_{1k} = \mu_{1k} x_{1k}, \quad \, k \in \hat{N}, \label{eq:2STOPDepot}\\
			\quad & y_{k} \in \{0, 1\}, \quad \, k \in \hat{N}, \label{eq:2STOPdef1}\\
			\quad & x_{ij} \in \{0,1\}, \quad \, (i,j) \in A \setminus \{(1,n)\}, \label{eq:2STOPdef2}\\
			\quad & z_{ij} \in \mathbb{R^+}, \quad \, (i,j) \in A,\label{eq:2STOPdef3}\\
  \quad & x_{1n} \in \mathbb{N}^+.\label{eq:2STOPdef4}
		\end{align}
	\end{subequations}
}


The objective function~\eqref{eq:2STOPof} maximises the recourse value, which corresponds to the expected profit collected by the selected routes defined as $\mathcal{Q}(x^{\nu}) = \sum_{r \in \mathcal{R}(x^\nu)} \mathcal{Q}(r)$. We denote with $\mathcal{R}(x^{\nu})$ the set of non-empty routes contained in the integer first-stage solution $x^{\nu}$ (an empty route consists solely of the arc $(1,n)$ and does not visit any customer).
Constraint~\eqref{eq:2STOPstartend} guarantees that each route starts from the source node and ends at the destination node.
Constraints~\eqref{eq:2STOPconnectivity} impose the connectivity of each route. Constraints~\eqref{eq:2STOPflow} and~\eqref{eq:2STOPUB} are the classical Gavish-Graves (GG) Subtour Elimination Constraints (SECs) adapted for the TOP~\citep{Gavish1978}. Constraints~\eqref{eq:2STOPLB} set the lower bound on the route expectation.
Constraints~\eqref{eq:2STOPDepot} bound the flow originating from the initial depot. Chance constraints~\eqref{eq:2chanceconstraints} eliminate all solutions containing routes associated with a route reliability less than $\alpha$. Finally,~\eqref{eq:2STOPdef1},~\eqref{eq:2STOPdef2},~\eqref{eq:2STOPdef3} and~\eqref{eq:2STOPdef4} define the domain of the variables.
%
%
%
%
To obtain a linear formulation, we relax~\eqref{eq:2chanceconstraints}, since these constraints are inherently non-linear. We refer to this formulation as $\mathscr{F}$. Feasibility cuts, introduced in Section~\ref{algorithm2}, are dynamically added to $\mathscr{F}$ during the resolution to ensure that the solution satisfies such chance constraints. The following lemma is stated:

\begin{lemma}
	\label{lem:lemma1}
	Constraints~\eqref{eq:2STOPflow} and~\eqref{eq:2STOPUB} guarantee that every route contained in a feasible integer solution of the first-stage problem has a route reliability greater than or equal to 0.5.
\end{lemma}

\begin{proof}
The constraints~\eqref{eq:2STOPflow} and~\eqref{eq:2STOPUB} imply that every integer solution satisfies $\mu_r \leq T_{\max}$ for each route $r$. Then, for any value of $\sigma^2_r$, the route reliability $\delta_r$ will be greater than or equal to 0.5 because the Normal distribution is symmetric with respect to its expectation.
\end{proof}

We also present an additional three-index mathematical formulation for the first-stage problem of the STOP that incorporates the chance constraints directly inside the model. However, due to its inferior computational performance compared to the formulation above, we have included it in the supplementary material.

\section{The Integer L-shaped method}
\label{sec:algorithm2}

We present in the following an ILM for solving formulation $\mathscr{F}$. We start by defining three new families of optimality cuts: (i) \emph{solution-based}, (ii) \emph{route-based}, and (iii) \emph{customer-based}. Next, we present the Upper-Bounding Functionals (UBFs), which operate analogously as Lower-Bounding Functionals, but are for maximisation problems. Moreover, we propose a set of upper bounds and valid inequalities. Finally, we introduce two constraints liftings to strengthen our first-stage formulation.
We note that the proofs of validity of the proposed optimality cuts and UBFs are reported in the supplementary material.

\subsection{Solution-based optimality cuts}
\label{ssec:solution-based}

In the solution-based optimality cuts, $\mathcal{Q}(x)$ is replaced by an auxiliary variable $\theta$, which represents the recourse of a first-stage solution.
Since each cut is defined for a specific first-stage solution, these optimality cuts must include the arcs connected to the source and destination nodes.
%
The following optimality cut for a first-stage solution $(x^{\nu}, \theta^\nu)$ is an adaptation of the one introduced by~\citet{Gendreau1995}:

\begin{equation}
\label{eq:optcuts1}
\theta \leq \mathcal{Q}(x^{\nu}) \overbrace{\left(\sum_{r \in \mathcal{R}(x^{\nu})} \left [ \sum_{(i,j) \in A(r)} x_{ij} - |A(r)|\right ] + 1 \right)}^{t_1} + \, \text{P}^m_{\text{UB}} \overbrace{\left(\sum_{r \in \mathcal{R}(x^{\nu})}\left [ |A(r)| - \sum_{(i,j) \in A(r)} x_{ij} \right ] \right)}^{t_2},
\end{equation}
where $\text{P}^m_{\text{UB}}$ is an upper bound on the expected profit of any first-stage solution. It is described in Section~\ref{ssec:bounds}.
%
Additionally, we introduce the following inequality to bound the problem:
\begin{equation}
\label{bound_sol_based_opt_cut}
\theta \leq \sum_{k \in \hat{N}} y_{k} \cdot p_k \cdot \mathbb{P}(\xi_{1k} + \xi_{kn} \leq T_{\max}).
\end{equation}
The right-hand side of inequality \eqref{bound_sol_based_opt_cut} provides an upper bound on the total expected profit assuming that each customer is visited by a single route.

We say that optimality cut~\eqref{eq:optcuts1} is active in solution $(x^\nu, \theta^\nu)$ if and only if $t_1 = 1$ and $t_2 = 0$. In this case, it follows that $\theta \leq \mathcal{Q}(x^{\nu})$.
Conversely, it does not bound $\theta$ (see the proof in the supplementary material). Finally, we state the following lemma.

\begin{lemma}
\label{lem:lemma2}
The value of $\theta^*$ is equal to $\mathcal{Q}(x^*)$ in any optimal solution $(x^*, \theta^*)$.
\end{lemma}

\begin{proof}
 Since the objective maximises $\theta$, only a specific optimality cut~\eqref{eq:optcuts1} associated with $x^*$ will be active in such an optimal solution. Consequently, we have that $\theta^* = \mathcal{Q}(x^*)$.
\end{proof}






\subsection{Route-based optimality cuts}
\label{ssec:route-based}

In the route-based optimality cuts, $\mathcal{Q}(x)$ is replaced by the sum $\sum_{k \in \hat{N}} \omega_k$, where each variable $\omega_k$ represents the recourse of a route that visits customer $k$ as first customer in a solution ($x_{1k}=1$). Since each cut is defined for a specific route, it has the advantage of being active for all solutions containing that route.
%
The following optimality cut for a route $r$ is formulated as an adaptation of the optimality cuts introduced by~\citet{Seguin1994}:

\begin{equation}
  \label{eq:optcuts2}
  \omega_{\pi(r)} \leq \mathcal{Q}(r)\overbrace{\left(\left [\sum_{(i,j) \in A(r)} x_{ij} - |A(r)| \right ] + 1 \right)}^{t_1} + \text{P}^1_{\text{UB}} \overbrace{\left( |A(r)| - \sum_{(i,j) \in A(r)} x_{ij} \right)}^{t_2},
\end{equation}
where $\text{P}^1_{\text{UB}}$ is an upper bound on the maximum expected profit achievable by a route described in Section~\ref{ssec:bounds} and $\pi(r)$ is the first customer visited by route $r$. To ensure the formulation remains bounded, each variable $\omega_k$ is restricted by the following inequality:
\begin{equation}
\omega_k \leq \text{P}^1_{\text{UB}} \cdot x_{1k}. \label{single_customer_bound1}
\end{equation}

We say that optimality cut~\eqref{eq:optcuts2} for route $r$ is active in solution $(x^{\nu}, \omega^{\nu})$ if and only if $t_1 = 1$ and $t_2 = 0$. In this case, it follows that $\omega_{\pi(r)} \leq \mathcal{Q}(r)$.
Conversely, it does not bound $\omega_{\pi(r)}$ (see the proof in the supplementary material).
%
Finally, we state the following lemma.
%

\begin{lemma}
  The value of $\sum_{k \in \hat{N}} \omega_k^*$ is equal to $\mathcal{Q}(x^*)$ in any optimal solution $(x^*, \omega^*)$.
\end{lemma}
\begin{proof}
  From inequalities~\eqref{single_customer_bound1}, it follows that $\omega^*_k = 0$ if customer $k$ is not the first customer visited by a route in $\mathcal{R}(x^*)$. The remaining variables may take a non-zero value.
  Since the objective maximises $\sum_{k \in \hat{N}} \omega_k$, only a specific optimality cut~\eqref{eq:optcuts2} associated with route $r \in \mathcal{R}(x^*)$ will be active, i.e. $\omega_{\pi(r)}^* = \mathcal{Q}(r)$. Summing over all $r \in \mathcal{R}(x^*)$, we have:
  \begin{equation*}
  \sum_{r \in \mathcal{R}(x^*)} \omega^*_{\pi(r)} = \sum_{r \in \mathcal{R}(x^*)} \mathcal{Q}(r) = \mathcal{Q}(x^*).
  \end{equation*}
\end{proof}

\subsection{Customer-based optimality cuts}
\label{ssec:customer-based}


The customer-based optimality cuts works similarly as the route-based ones, but they are defined for paths of customers, and therefore can exclude the arcs associated with the source and the destination.
Similarly to the previous optimality cuts presented, $\mathcal{Q}(x)$ is replaced by a sum $\sum_{k \in \hat{N}} \lambda_k$ of auxiliary variables, where $\lambda_k$ represents, in some way, the contribution to the recourse of customer $k \in \hat N$.
%
%
%
The following optimality cut for a customers path $w$ is formulated as an adaptation of the P-cuts introduced by~\citet{Lucas2024}:
\begin{equation}
  \label{eq:optcuts3}
  \sum_{k \in N(w)} \lambda_k \leq \mathcal{Q}(r_w) \overbrace{\Biggl(\left [\sum_{(i,j) \in A(w)} x_{ij} - |A(w)|\right ] + 1 \Biggl)}^{t_1} + \, \text{P}_{N(w)} \overbrace{\Biggl( |A(w)| - \sum_{(i,j) \in A(w)} x_{ij} \Biggl)}^{t_2},
\end{equation}
where $r_w$ is the route that performs path $w$, and let $\text{P}_{N(w)}$ be an upper bound on the recourse for serving the customers $N(w)$, which is described in Section~\ref{bounds:set-upper-bound}. Finally, we add the following constraint for each customer $k \in \hat N$ to have a bounded formulation:
\begin{equation}
\lambda_k \leq y_{k} \cdot p_k \cdot \mathbb{P}(\xi_{1k} + \xi_{kn} \leq T_{\max}). \label{single_customer_bound2}
\end{equation}

The right-hand side of inequality~\eqref{single_customer_bound2} provides an upper bound on $\lambda_k$ and represents the expected profit of a route serving only customer $k$. We say that inequality~\eqref{single_customer_bound2} is active whenever $y_k = 1$.

We say that optimality cut~\eqref{eq:optcuts3} for path $w$ is active in solution $(x^\nu, \lambda^\nu)$ if and only if $t_1 = 1$ and $t_2 = 0$. In this case, it follows that $\sum_{k \in N(w)} \lambda_k \leq \mathcal{Q}(r_w)$.
Conversely, it does not bound $\sum_{k \in N(w)} \lambda_k$ (see the proof in the supplementary material).
Since each cut is defined for a specific path, it has the advantage of being active for all solutions containing such a path, which can bound a broader range of solutions than the route-based optimality cuts.
%

\cite{Legault2025} have shown that a DILM is valid if the recourse function satisfies the superadditivity property. Let $\mathcal{R}$ be the set of feasible routes for the problem. Also, let $(q,h)$ denote the customers path obtained by concatenating paths $q$ and $h$ and let $r=r_{(q,h)}$, $r_q$ and $r_h$ be the routes that perform the corresponding customers paths. The superadditivity property for the STOP is stated as follows.

\begin{property}
  \label{prop_super_add}
  The recourse function is superadditive if $\mathcal{Q}(r) \leq \mathcal{Q}(r_q) + \mathcal{Q}(r_h)$.
\end{property}


\begin{lemma}
  \label{lem:superadditivity}
  The superadditivity property holds for the recourse function $\mathcal{Q}(r)$ evaluated on route $r$.
\end{lemma}

\begin{proof}
  Let $q=(i_1,...,i_t), h=(i_{t+1},...,i_l)$ be two customers paths with $r = r_{(q, h)} \in \mathcal{R}$. Writing the superadditivity condition for route $r$ gives:
  \begin{align}
  \delta_r \sum_{k \in N(r)}p_k & \leq \delta_{r_q} \sum_{k \in N(q)}p_k + \delta_{r_h} \sum_{k \in N(h)}p_k.
  \label{eq:superadditivity}
  \end{align}
  Given that $\mu_{ij}$ and $\sigma^2_{ij}$ respect the triangular inequality, we have that:
  \begin{equation*}
  \mu_{1i_1} + \sum_{k=1}^{t-1} \mu_{i_k i_{k+1}} + \mu_{i_t n} \leq \mu_{1i_1} + \sum_{k=1}^{l-1} \mu_{i_k i_{k+1}} + \mu_{i_l n}, \quad \mu_{1i_{t+1}} + \sum_{k=t+1}^{l-1} \mu_{i_k i_{k+1}} + \mu_{i_l n} \leq \mu_{1i_1} + \sum_{k=1}^{l-1} \mu_{i_k i_{k+1}} + \mu_{i_l n},
  \end{equation*}
  and
  \begin{equation*}
  \sigma^2_{1i_1} + \sum_{k=1}^{t-1} \sigma^2_{i_k i_{k+1}} + \mu_{i_t n} \leq \sigma^2_{1i_1} + \sum_{k=1}^{l-1} \sigma^2_{i_k i_{k+1}} + \sigma^2_{i_l n}, \quad \sigma^2_{1i_{t+1}} + \sum_{k=t+1}^{l-1} \sigma^2_{i_k i_{k+1}} + \sigma^2_{i_l n} \leq \sigma^2_{1i_1} + \sum_{k=1}^{l-1} \mu_{i_k i_{k+1}} + \sigma^2_{i_l n}.
  \end{equation*}
  Thus:
  \[
  \delta_r \leq \delta_{r_q} \quad \text{and} \quad \delta_r \leq \delta_{r_h},
  \]
  We can multiply each inequality by the sum of profits of each path, giving:
  \[
  \delta_r \sum_{k \in N(q)} p_k \leq \delta_{r_q} \sum_{k \in N(q)} p_k \quad \text{and} \quad \delta_r \sum_{k \in N(h)} p_k \leq \delta_{r_h} \sum_{k \in N(h)} p_k.
  \]
  Summing both inequalities gives~\eqref{eq:superadditivity}.
\end{proof}

\noindent Finally, we state the following lemma.

\begin{lemma}
 The value of $\sum_{k \in \hat{N}} \lambda_k$ is equal to $\mathcal{Q}(x^*)$ in any optimal solution $(x^*, \lambda^*)$.
 \label{lem:optcuts3}
\end{lemma}

\begin{proof}
  From inequalities~\eqref{single_customer_bound2}, it follows that $\lambda^*_k = 0$ if customer $k$ is not visited by any route in $\mathcal{R}(x^*)$. The remaining variables may take a non-zero value. For each route $r \in \mathcal{R}(x^*)$, we denote with $\mathcal{P}_r$ the collection of customers sub-paths of $r$ associated with active inequalities~\eqref{eq:optcuts3} and~\eqref{single_customer_bound2}.
  From each set $\mathcal{P}_r$, we construct the corresponding set $\Omega_r$ as follows:
  \[
  \Omega_r = \left\{ P \subseteq \mathcal{P}_r : \bigcup_{\substack{h \in P}} N(h) = \hat N(r),\ \bigcap_{\substack{h \in P}} N(h) = \emptyset \right\},
  \]
  which represents the set of partitions obtained by aggregating paths of $\mathcal{P}_r$.
  The superadditivity property stated for an arbitrary route $r$ can be easily generalised to any collection of customers paths $P$ such that sets $N(h)$, for each $h \in P$, form a partition of $\hat N(r)$.
  Since the recourse function $\mathcal{Q}(r)$ satisfies the superadditivity Property~\ref{prop_super_add} for route $r$, it holds that $\mathcal{Q}(r) \leq \sum_{h \in P} \mathcal Q(r_h)$ for each $P \in \Omega_r$.
  %
  Consequently, $\sum_{k \in \hat{N}(r)} \lambda_k^* = \mathcal{Q}(r)$ for each $r \in \mathcal{R}(x^*)$.
  Summing over all routes:
  \begin{align*}
  \sum_{k \in N} \lambda_k^* =
  \sum_{r \in \mathcal{R}(x^*)}
  \sum_{k \in \hat N(r)} \lambda_k^* =
  \sum_{r \in \mathcal{R}(x^*)}
  \mathcal{Q}(r) =
  \mathcal{Q}(x^*).
  \end{align*}
\end{proof}

\subsection{Upper-Bounding Functionals}
\label{ssec:ubf}

This section introduces the Upper-Bounding Functionals (UBFs), which are an adaptation of the Lower-Bounding Functionals proposed by~\cite{Lucas2024}, originally developed for minimisation problems. The purpose of such inequalities is to tighten the value of the objective function by computing a specific upper bound on the recourse function of a generic route that visits a given set of customers.
Given a set of customers $S$, an UBF is formulated as follows:

\begin{equation}
	\label{eq:ubf}
	\sum_{k \in S} \lambda_k \leq \mathcal{Q}_{\text{UB}}(S) \overbrace{\left(\left [ \sum_{(i,j) \in S \times S} x_{ij} - |S| \right ] + 2 \right)}^{t_1} + \, \text{P}_{\text{S}} \overbrace{\left(|S| - 1 - \sum_{(i,j) \in S \times S} x_{ij}\right)}^{t_2}.
\end{equation}

In the above inequality, $\text{P}_{\text{S}}$ is the Set-based upper bound on $S$ introduced in Section~\ref{bounds:set-upper-bound} and $\mathcal{Q}_{\text{UB}}(S)$ is an upper bound on the recourse of any arbitrary route that visits all customers in $S$.

We say that UBF~\eqref{eq:ubf} for set $S$ is active in solution $(x^{\nu}, \lambda^{\nu})$ if and only if $t_1 = 1$ and $t_2 = 0$. In this case, it follows that $\sum_{k \in S} \lambda_k \leq \mathcal{Q}_{\text{UB}}(S)$.
Conversely, it does not bound $\sum_{k \in S} \lambda_k$ (see the proof in the supplementary material). They are defined only for the customer-based optimality cuts and differ from the corresponding cuts since they can be active for multiple solutions sharing the same set of customers.

Since determining the route with the maximum route reliability value that visits all the customers in $S$ is equivalent of solving a Travelling Salesman Problem (TSP) with stochastic travel times, it follows that finding such a route is NP-Hard. For efficiency reasons, we compute an upper bound $\delta_{\text{UB}}$ for the route reliability.

To compute $\delta_{\text{UB}}$, we define two directed subgraphs $G_{\mu}$ and $G_{\sigma^2}$ over set $S$.
In $G_{\mu}$, the arc weights correspond to the arc expectations $\mu_{ij}$, while in $G_{\sigma^2}$, they represent the arc variances $\sigma^2_{ij}$. Next, we compute a lower bound on the route expectation and a lower bound on the route variance for any arbitrary route that visits all customers in $S$. The first bound, $H^{\mu}_{\mathrm{LB}}$, is computed on $G_{\mu}$, while the second, $H^{\sigma^2}_{\mathrm{LB}}$, is computed on $G_{\sigma^2}$. Both lower bounds are obtained by applying the Helsgaun bound~\citep{Helsgaun2000} on its corresponding subgraph.
The Helsgaun bound, as well as the Held-Karp bound~\citep{Held1970}, uses the Lagrangian relaxation of the classical TSP formulation to compute a lower bound of the TSP through the 1-tree structure which consists of a minimum spanning tree with one additional edge. We employed the subgradient method to solve the Lagrangian dual problem by using a simple decreasing step size and modifying the Lagrangian multipliers in accordance with the approach outlined by~\cite{Volgenant1982}.

Let us define a Normal random variable $\xi_{\text{LB}} \sim \mathcal{N}(H^{\mu}_{\text{LB}}, H^{\sigma^2}_{\text{LB}})$ which provides a lower estimation on the route uncertain travel time. By evaluating its cumulative distribution function on $T_{\max}$, we can compute an upper bound $\delta_{\text{UB}}$ on the reliability of any arbitrary route that visits all customers in $S$:
\[\delta_{\text{UB}} = \mathbb{P}(\xi_{\text{LB}} \leq T_{\max}).\]
Finally, $\mathcal{Q}_{\text{UB}}(S)$ can be calculated by:
\[\mathcal{Q}_{\text{UB}}(S) = \delta_{\text{UB}} \cdot \sum_{k \in S} p_k.\]
Therefore, $\mathcal{Q}_{\text{UB}}(S)$ represents an upper bound on the recourse value of any arbitrary route that visits all customers in $S$.

\subsection{Upper bounds}
\label{ssec:bounds}

This section presents a set of upper bounds divided into: (i) \emph{customers upper bounds} and (ii) \emph{profit upper bounds}. Each of them is further subdivided in: (i) \emph{single-route} and (ii) \emph{multi-route}. Let $\text{C}^1_{\text{UB}}$ and $\text{P}^1_{\text{UB}}$ denote the upper bounds on the number of customers and on the expected profit for a single route, respectively. Let $\text{C}^m_{\text{UB}}$ and $\text{P}^m_{\text{UB}}$ be the same bounds considering $m$ routes instead of only one.

All such bounds are calculated by solving the Continuous Stochastic Team Orienteering Problem (CSTOP) that consists in the continuous relaxation of the STOP. In the CSTOP, we relax the chance constraints~\eqref{eq:2chanceconstraints} and dynamically separate the feasibility cuts presented in Section~\ref{algorithm2} to ensure feasibility with respect to them. The CSTOP shares the same objective of the STOP, and the convergence to an optimal fractional solution is guaranteed by the separation of the optimality cuts.
%
%
Essentially, we consider two different variants of the CSTOP. When computing $\text{P}^1_{\text{UB}}$ and $\text{P}^m_{\text{UB}}$, we retain the original objective function. Conversely, when computing $\text{C}^1_{\text{UB}}$ and $\text{C}^m_{\text{UB}}$, the objective function is modified to maximise the number of served customers; in this case, we do not separate the optimality cuts. We update these bounds iteratively by solving the two variants of the CSTOP using the current bounds as input parameters. The process iterates until convergence.

\subsubsection{Set-based upper bound.}
\label{bounds:set-upper-bound}
We also introduce $P_S$ for the set of customers $S$ as an upper bound on the recourse of any arbitrary route that visits all customers in $S$:
\begin{equation}
\label{eq:set-ub}
P_S = \sum_{k \in S} \mathbb{P}(\xi_{1k} + \xi_{kn} \leq T_{\max}) \cdot p_k.
\end{equation}
Its validity is demonstrated in Lemma~\ref{lem:lemma7}.

\begin{lemma}
\label{lem:lemma7}
$\text{P}_{\text{S}}$ is a valid upper bound on the recourse of any arbitrary route that visits all customers in $S$.
\end{lemma}

\begin{proof}
Let $r$ be a route that visits all the customers in $S$ in an arbitrary order.
As discussed in Lemma~\ref{lem:optcuts3}, the superadditivity property of the recourse function~\ref{lem:superadditivity} for route $r$ can be stated for any partition $P$ of customers paths of $\hat{N}(r)$.
If $P$ is a collection of paths such that each element is a unit set containing exactly one customer of $\hat{N}(r)$, then $P_{\text{S}} \geq \mathcal{Q}(r)$, which concludes the proof.
%
\end{proof}

\subsection{Valid inequalities}
\label{ssec:vi}

\noindent This section presents three families of valid inequalities to tighten the value of the objective function.

\paragraph{Solution-based optimality cuts.} To bound the value of $\theta$, we add:

\[\theta \leq \text{P}^m_{\text{UB}}.\]

\paragraph{Route-based optimality cuts.} To bound the value of $\sum_{k \in \hat{N}} \omega_k$, we add:

\[\sum_{k \in \hat{N}} \omega_k \, \leq \sum_{k \in \hat{N}} y_{k} p_k \, \mathbb{P}(\xi_{1k} + \xi_{kn} \leq T_{\max}) \quad \text{and} \quad \sum_{k \in \hat{N}} \omega_k \leq \text{P}^m_{\text{UB}}.\]

\paragraph{Customer-based optimality cuts.} To bound the value of $\sum_{k \in \hat{N}} \lambda_k$, we add:

\[\sum_{k \in \hat{N}} \lambda_k \leq \text{P}^m_{\text{UB}}.\]
The mathematical formulation of the first-stage problem~\eqref{eq:2STOPof}--\eqref{eq:2STOPdef4} is further strengthened by adding an additional inequality to impose a threshold on the maximum number of served customers in a solution:

\[\sum_{k \in \hat{N}} y_{k} \leq \text{C}^m_{\text{UB}}.\]



\subsection{Contraints liftings}
\label{ssec:cl}

In this section, we introduce two constraints liftings to strengthen the formulation of the first-stage problem~\eqref{eq:2STOPof}--\eqref{eq:2STOPdef4}. The first, the Time-budget lifting, tightens the value of $T_{\max}$ in~\eqref{eq:2STOPUB}. The second, the Sedrakyan lifting, strengthens constraints~\eqref{eq:2STOPflow},~\eqref{eq:2STOPUB},~\eqref{eq:2STOPLB}, and~\eqref{eq:2STOPDepot}.

\subsubsection{Time budget lifting}
The idea behind this lifting is to determine the maximum route expectation value that a route can attain.
To this end, the value of $T_{\max}$ in~\eqref{eq:2STOPUB} is progressively tightened without cutting off any feasible solution. This is achieved by alternating the computation of a lower bound on the route variance $\sigma^2_{\text{LB}}$ and the computation of an upper bound on the route expectation $\mu_{\text{UB}}$.
To compute $\sigma^2_{\text{LB}}$, we solve an adaptation of the single-route CSTOP following the same approach described in Section~\ref{ssec:bounds}. In this case, the CSTOP objective function was modified to minimise $\sum_{(i,j) \in A} \sigma^2_{ij} \, x_{ij}$, with an additional constraint requiring that the route expectation must be greater than or equal to $\mu_{\text{UB}}$.
The value of $\mu_{\text{UB}}$ is then updated by taking the (${1-\alpha}$)-quantile of the Normal distribution $\mathcal{N}(T_{\max}, \sigma^2_{\text{LB}})$. In other words, $\sigma^2_{\text{LB}}$ and $\mu_{\text{UB}}$ are updated iteratively until convergence is achieved; in our case, this occurs when the difference between the value at the previous iteration and the value at the current iteration is close to zero for both parameters (initially, $\mu_{\text{UB}} = T_{\max}$).
The pseudocode of the procedure is reported in Algorithm~\ref{algo:tbl} and makes use of the following procedures:
\begin{inparaenum}[(i)]
	\item \texttt{RVLB}: calculates the lower bound on the route variance given a specified route expectation $\mu_{\text{UB}}$; and
	\item \texttt{quantile}: calculates the specified quantile for the given distribution.
\end{inparaenum}

\allowdisplaybreaks{
	\begin{algorithm}[!ht]
		\KwData{$T_{\max}$, $\alpha$}
		\KwResult{$\mu_{\text{UB}}$}
		$\sigma^2_{\text{LB}} \leftarrow +\infty$; $\mu_{\text{UB}} \leftarrow T_{\max}$\;
		\While{not \texttt{convergence} $(\mu_{\text{UB}}, \sigma^2_{\text{LB}})$}{
			$\sigma^2_{\text{LB}} \leftarrow$ \texttt{RVLB} $(\mu_{\text{UB}})$\;
			$\mu_{\text{UB}} \leftarrow$ \texttt{quantile} ($\mathcal{N}(T_{\max}$, $\sigma^2_{\text{LB}})$, 1 - $\alpha$)\;
		}

		\Return{$\mu_{\text{UB}}$;}
		\caption{\texttt{The Time budget lifting} procedure.}
		\label{algo:tbl}
	\end{algorithm}
}
The validity of $\mu_{\text{UB}}$ is demonstrated in Lemma~\ref{lem:lemma6}:
\begin{lemma}
\label{lem:lemma6}
$\mu_{\text{UB}}$ is a valid upper bound on the route expectation.
\end{lemma}

\begin{proof}
By definition, $\mu_{\text{UB}}$ represents the (${1-\alpha}$)-quantile of the Normal distribution $\bar{\xi} \sim \mathcal{N}(T_{\max}, \sigma^2_{\text{LB}})$:
\[
\mathbb{P}(\bar{\xi} \leq \mu_{\text{UB}}) = 1 - \alpha.
\]
Since the normal distribution is symmetric with respect to its mean, we have:
\[
\mathbb{P}(\bar{\xi} \leq 2 \, T_{\max} - \mu_{\text{UB}}) = \alpha.
\]
Subtracting the quantity $T_{\max} - \mu_{\text{UB}}$ from both sides inside the probability term, it follows that:
\begin{equation}
\label{eq:quantile}
\mathbb{P}(\bar{\xi} - (T_{\max} - \mu_{\text{UB}}) \leq T_{\max}) = \alpha,
\end{equation}
where the mean of the Normal random variable $\bar{\xi} - (T_{\max} - \mu_{\text{UB}})$ is equal to $\mu_{\text{UB}}$.
Therefore, for any route $r$ with a route uncertain travel time $\xi_r \sim \mathcal{N}(\mu_r, \sigma^2_{\text{LB}})$ such that $\mu_r \leq \mu_{\text{UB}}$, it holds that:
\[
\delta_r \geq \alpha,
\]
which means that $r$ is feasible.
Thus, the probability $\delta_r$ achieves the minimum value when $\mu_r = \mu_{\text{UB}}$, where $\delta_r = \alpha$ as in~\eqref{eq:quantile}. Consequently, there is no route $r$ that satisfies:
\begin{align*}
\mu_r > \mu_{\text{UB}} \, \text{ and } \delta_r \geq \alpha,
\end{align*}
considering a route variance equal to $\sigma^2_{\text{LB}}$. This implies that $\mu_{\text{UB}}$ represents a valid upper bound on the route expectation.
\end{proof}

\subsubsection{Sedrakyan lifting}
The aim of this lifting is to strengthen the expressions~\eqref{eq:2STOPflow}, \eqref{eq:2STOPUB}, \eqref{eq:2STOPLB} and \eqref{eq:2STOPDepot}. Each route $r$ contained in a feasible first-stage solution must satisfy all the chance constraints~\eqref{eq:2chanceconstraints} which, according to~\cite{Hillier1967}, can be rephrased into the following non-linear constraints:

\begin{equation}
	\sum_{(i,j) \in A(r)} \mu_{ij} x_{ij} + q^{\mathcal{N}}_{\alpha} \sqrt{\sum_{(i,j) \in A(r)} \sigma^2_{ij} x_{ij}} \leq T_{\max},
	\label{eq:nlcc}
\end{equation}

where $q^{\mathcal{N}}_{\alpha}$ represents the $\alpha$-quantile of a standard Normal distribution $\mathcal{N}(0, 1)$ and is a non-negative quantity for $\alpha \geq 0.5$. The idea is to linearly relax inequality~\eqref{eq:nlcc} through the so-called Sedrakyan's inequality~\citep{Sedrakyan2018}:
\begin{equation*}
	\sum_{k=1}^t \frac{a_k^2}{b_k} \geq \frac{\left(\sum\limits_{k=1}^t a_k\right)^2}{\sum\limits_{k=1}^t b_k} \quad \, t \in \mathbb{N}^+, \quad a_1, \ldots, a_t \in \mathbb{R}, \quad b_1, \ldots, b_t \in \mathbb{R}^+.
\end{equation*}
Imposing $t = |A(r)|$, $b_k = 1$ and $a_k = \sigma_{ij} x_{ij}$, it follows that:
\begin{equation}
	\sum_{(i,j) \in A(r)} \sigma^2_{ij} x_{ij} \geq \frac{\left(\sum\limits_{(i,j) \in A(r)} \sigma_{ij} x_{ij}\right)^2}{|A(r)|}.
  \label{sed1}
\end{equation}
Taking the square root of both sides:
\begin{equation}
	\sqrt{\sum_{(i,j) \in A(r)} \sigma^2_{ij} x_{ij}} \geq \frac{\sum\limits_{(i,j) \in A(r)} \sigma_{ij} x_{ij}}{\sqrt{|A(r)|}}.
  \label{sed2}
\end{equation}
Then, constraint in~\eqref{eq:nlcc} can be relaxed by replacing the square root term in~\eqref{eq:nlcc} with the right-hand side of~\eqref{sed2}, giving:
\begin{equation*}
	\sum_{(i,j) \in A(r)} \mu_{ij} x_{ij} + q^{\mathcal{N}}_{\alpha} \, \frac{\sum\limits_{(i,j) \in A(r)} \sigma_{ij} x_{ij}}{\sqrt{|A(r)|}} \leq T_{\max}.
\end{equation*}
%
%
Aggregating the coefficients associated with the the $x_{ij}$ variables, we have:

\begin{equation}
	\sum_{(i,j) \in A(r)} \left(\mu_{ij} + q^{\mathcal{N}}_{\alpha} \, \frac{\sigma_{ij}}{\sqrt{|A(r)|}}\right) x_{ij} \leq T_{\max},
  \label{eq:lin_cc}
\end{equation}

Finally, we can replace~\eqref{eq:2STOPflow}, \eqref{eq:2STOPUB}, \eqref{eq:2STOPLB}, and \eqref{eq:2STOPDepot} with the following inequalities, respectively. In these constraints, the term $|A(r)|$ has been replaced with $(C_{\text{UB}}^1 + 1)$. Since $C_{\text{UB}}^1$ bounds the number of customers that a route can visit and the number of arcs in a route equals the number of visited nodes plus one, $(C_{\text{UB}}^1 + 1)$ provides a valid upper bound on the number of arcs in a single route.
\begin{align*}
	\quad & \sum_{(k,j) \in \hat{A}} z_{kj} - \sum_{(i,k) \in \hat{A}} z_{ik} = \sum_{(k,j) \in \hat{A}} \left( \mu_{kj} + q^{\mathcal{N}}_{\alpha} \frac{\sigma_{kj}}{\sqrt{(C_{\text{UB}}^1 + 1)}} \right) x_{kj}, \quad \, k \in \hat{N},\\
	\quad & z_{ij} \leq \left(T_{\max} - \mu_{jn} - q^{\mathcal{N}}_{\alpha} \frac{\sigma_{jn}}{\sqrt{(C_{\text{UB}}^1 + 1)}}\right) \, x_{ij}, \quad \, (i,j) \in \hat{A},\\
	\quad & z_{ij} \geq \left(\mu_{1i} + \mu_{ij} + q^{\mathcal{N}}_{\alpha} \frac{\sigma_{1i} + \sigma_{ij}}{\sqrt{(C_{\text{UB}}^1 + 1)}}\right) \, x_{ij}, \quad \, (i,j) \in \hat{A},\\
	\quad & z_{1k} = \left(\mu_{1k} + q^{\mathcal{N}}_{\alpha} \frac{\sigma_{1k}}{\sqrt{(C_{\text{UB}}^1 + 1)}} \right) x_{1k}, \quad \, k \in \hat{N},
\end{align*}

\subsection{The algorithm}
\label{algorithm2}

This section presents our implementation of the ILM for solving the formulation $\mathscr{F}$. We refer to it as the two-index Integer L-shaped Method (2ILM) since it is based on the two-index mathematical formulation of the first-stage problem introduced in Section~\ref{sec:problem}.

The seperation algorithms have been developed inside a mixed-integer solver's environment as a callback function that is called for each node of the branch-and-bound tree. At each call, we get a solution ($\bar{x}$, $\bar{y}$, $\bar{z}$) and we build a graph $\mathcal{G}(\bar{x}) = (N, \bar{A})$, where set $\bar{A} = \{(i, j) \in \hat{A} : \bar{x}_{ij} > 0\}$ includes only the arcs of positive value in the solution. The set of routes $\mathcal{R}(\bar{x})$ are enumerated using a Depth First Search (DFS) algorithm. Next, we perform the separation of feasibility cuts, optimality cuts, Upper-Bounding Functionals, and subtour elimination constraints. We detail their separation as follows.



\paragraph{Feasibility cuts.} For each $r=\{1, i_1,\ldots, i_k, n\} \in \mathcal{R}(\bar{x})$, we start by computing $\delta_r$ and $\delta_{UB}$. If $\delta_{UB} < \alpha$, which indicates that no feasible route can visit the customers of $\hat{N}(r)$, then we add inequality \eqref{setIneq} if violated. The inequality imposes that the set of nodes in $\hat{N}(r)$ must be visited by at least two routes.

\begin{equation}
\label{setIneq}
  \sum_{i \in \hat{N}(r)} \sum_{j \in \hat{N}(r)} x_{ij} \leq |\hat{N}(r)| - 2.
\end{equation}

Otherwise, if route $r$ is infeasible ($\delta_r < \alpha$), then we had inequality \eqref{routeIneq} to forbid the route.

\begin{equation}
\label{routeIneq}
\sum_{(i,j) \in \hat{A}(r)} x_{ij} \leq \sum_{i \in \hat{N}(r)\setminus\{i_1,i_k\}} y_i.
\end{equation}

We can strengthen inequality \eqref{routeIneq} by considering feasible sub-paths of $r$. Let $w=(i_u,\ldots,i_v)$ be a sub-path of $r$, with $1 \leq u < v \leq k$, such that $r_w$ is feasible. Let sets $L(w)$ and $R(w)$ represent the set of nodes that can be added without violating any chance constraint before $u$ and at the end of $w$, respectively. Then, the following inequalities can be added:
\begin{align}
\sum_{(i,j) \in A(w)} x_{ij} & \leq \sum_{k \in N(w)\setminus\{i_u,i_v\}} y_{k} + \sum_{j \in L(w)} x_{j i_u}. \label{subpathIneq1} \\
\sum_{(i,j) \in A(w)} x_{ij} &\leq \sum_{k \in N(w)\setminus\{i_u,i_v\}} y_{k} + \sum_{j \in R(w)} x_{i_v j}. \label{subpathIneq2}
\end{align}

The inequality \eqref{subpathIneq1} requires that $w$ be connected on the left with with the nodes $L(w)$ to be feasible. Similarly, inequality \eqref{subpathIneq2} requires that $w$ be connected on the right with with the nodes $R(w)$ to be feasible. Note that the source node 1 is part of $L(w)$, and $n$ is part of $R(w)$. The separation of these inequalities is done for each customer sub-path $w$ of $r$.

\paragraph{Optimality cuts.}
We use one of the three families optimality cuts (Section~\ref{sec:algorithm2}). For solution-based optimality cuts, we add the cut associated with the entire first-stage solution if violated. For route-based optimality cuts, we add the cuts associate with each route in the solution whenever they are violated. For customer-based optimality cuts, we add the $\phi = 3$ most violated cuts of each route $r \in \mathcal{R}(\bar{x})$.

\paragraph{Upper-Bounding Functionals.}
For each $r \in \mathcal{R}(\bar{x})$, we compute the upper bound $\mathcal{Q}_{\text{UB}}(r)$, and add an inequality inequality~\eqref{eq:ubf} if violated.

\paragraph{Subtour elimination constraints.}
All the elementary cycles inside $\mathcal{G}(\bar{x})$ are computed~\citep{Hawick2008} and the associated SECs are checked for violation:
\begin{equation}
\label{eq:2secs}
\sum_{i \in U} \sum_{j \in U} x_{ij} \leq \sum_{i \in U} y_i - y_k, \quad \, U \subseteq \hat{N},\; k \in U.
\end{equation}

\allowdisplaybreaks{
	\begin{algorithm}
		\KwData{$T_{\max}$, $\mu_{ij}$, $\sigma^2_{ij}$, with $i,j \in \{1, \dots, n\}$, $\bar{x}$}
		$\mathcal{G}(\bar{x}) \leftarrow$ \texttt{buildGraph}($\bar{x}$)\;
		$\mathcal{R}(\bar{x}) \leftarrow$ \texttt{buildRoutes}($\mathcal{G}(\bar{x})$, $1$, $n$)\;

		\For{route $r \in \mathcal{R}(\bar{x})$}{
			$G_{\mu}, G_{\sigma^2} \leftarrow$ \texttt{buildSubgraphs}($r$, $\mu_{ij}$, $\sigma^2_{ij}$)\;
			$H^{\mu}_{\text{LB}}, H^{\sigma^2}_{\text{LB}} \leftarrow$ \texttt{calculateProbBounds}($G_{\mu}, G_{\sigma^2}$)\;
			$\xi_{\text{LB}} \leftarrow \mathcal{N}\big(H^{\mu}_{\text{LB}}, H^{\sigma^2}_{\text{LB}}\big)$\;
			$\delta_{\text{UB}} \leftarrow \mathbb{P}(\xi_{\text{LB}} \leq T_{\max})$\;

			\If{$\delta_r < \alpha$}{
  \For(\tcc*[f]{Feasibility cuts}){sub-path $w \in r$} {
					\If{$1 \notin w$ and $n \notin w$}{
						\If{$\delta_{r_w} \geq \alpha$}{
							$L(w), R(w) \leftarrow$ \texttt{buildSets}($w$, $N$, $T_{\max}$)\;
							check inequality~\eqref{subpathIneq1} associated with $w$ and $L(w)$\;
							check inequality~\eqref{subpathIneq2} associated with $w$ and $R(w)$\;
						}
					}
				}

				\If{$\delta_{\text{UB}} < \alpha$}{
					check inequality~\eqref{setIneq} associated with $\hat{N}(r)$\;
				}
				\Else
				{
					check inequality~\eqref{routeIneq} associated with $r$\;
				}
			}
			\Else
			{
				$\mathcal{Q}_{\text{UB}}(r) \leftarrow \delta_{\text{UB}} \cdot \sum_{k \in \hat{N}(r)}p_k$\;
  				check $\phi$ optimality cuts associated with $r$\tcc*{Optimality cuts}

  			check the upper-bounding functional associated with $\hat{N}(r)$\tcc*{UBFs}
  			}
  		}

		\For(\tcc*[f]{Subtour elimination constraints}){cycle $c \in \mathcal{G}(\bar{x})$}{
			check inequalities~\eqref{eq:2secs} associated with $c$\;
		}

		\caption{\texttt{2ILM} callback.}
		\label{algo:2ILM}
	\end{algorithm}
}
%

\paragraph{Valid inequalities.}
One of the three families of valid inequalities introduced in Section~\ref{ssec:vi} is inserted into the formulation of the first-stage problem.

\paragraph{Constraints Liftings.}
One of the two constraints liftings reported in Section~\ref{ssec:cl} is adopted.

\subsubsection{Implementation}
The pseudocode of the procedure callback is reported in Algorithm~\ref{algo:2ILM} and makes use of the following procedures:
\begin{inparaenum}[(i)]
  \item \texttt{buildGraph}: builds the graph $\mathcal{G}(\bar{x})$ from solution $\bar{x}$;
	\item \texttt{buildRoutes}: builds the set $\mathcal{R}(\bar{x})$ of routes from the graph $\mathcal{G}(\bar{x})$ with the DFS algorithm;
	\item \texttt{buildSubgraphs}: builds the subgraphs $G_{\mu}$ and $G_{\sigma^2}$ from sets $N(r)$ and $A(r)$;
	\item \texttt{calculateProbBounds}: calculate the Helsgaun lower bounds on the subgraphs $G_{\mu}$ and $G_{\sigma^2}$; and
	\item \texttt{buildSets}: build the sets $L(w)$ and $R(w)$ from the customers sub-path $w$.
\end{inparaenum}

\bigskip
\noindent
Further details about the separation methods adopted for the feasibility cuts and subtour elimination constraints are reported~\cite{Guastalla2025}.

We finally report, in the supplementary material, a three-index adaptation of the 2ILM, which we call the three-index Integer L-shaped (3ILM), based on the three-index formulation presented therein.



\section{Computation experiments}
\label{sec:analysis}

This section presents a quantitative assessment of the proposed formulation and algorithm. We begin by reporting the results of preliminary experiments designed to evaluate the effect of the UBFs, followed by a detailed analysis of the algorithm’s computational performance, focusing on solution quality, computational time, and scalability. Additionally, the supplementary material presents a comparison between the two-index and three-index formulations for the first-stage problem of the STOP. It is worth noting that, in these experiments, the solver was provided with a lower bound computed by a simple metaheuristic to improve efficiency, and parameter $\alpha$ was set to 0.95 to ensure that only highly reliable routes were considered. Finally, we conclude the section with an analysis of the impact of parameter $\alpha$ under different values.
\begin{center}
\begin{tabular}{rrrrrrrrr}
\toprule
\multicolumn{3}{c}{\emph{Small}}
& \multicolumn{3}{c}{\emph{Medium}}
& \multicolumn{3}{c}{\emph{Large}} \\
\cmidrule(lr){1-3} \cmidrule(lr){4-6} \cmidrule(l){7-9}
\textbf{Set} & \textbf{$|N|$} & \textbf{$|A|$}
& \textbf{Set} & \textbf{$|N|$} & \textbf{$|A|$}
& \textbf{Set} & \textbf{$|N|$} & \textbf{$|A|$} \\
\midrule
1 & 32  & 992
& 5 & 66  & 4290
& 4 & 100 & 9900  \\

2 & 21  & 420
& 6 & 64  & 4032
& 7 & 102 & 10302 \\

3 & 33  & 1056
& - & - & -
& - & - & - \\

\bottomrule
\end{tabular}
\captionof{table}{A summary of the TOP instances benchmark sets.}
\label{tab:instances}
\end{center}
All computational results are based on the TOP reference dataset available online at \url{https://www.mech.kuleuven.be/en/cib/op}.
This benchmark set, proposed by~\cite{CHAO1996464}, is composed of 387 instances divided into 7 sets, with the number of nodes ranging from 21 to 102. Each set differs in the number of available routes used to serve customers (from 2 to 4) and in the value of the maximum route duration. The instances consist of nodes represented by two-dimensional coordinates in a Euclidean plane. To present the results, the TOP instances were grouped into three benchmark categories.
The \emph{small} benchmark set is composed of the instances belonging to sets 1, 2 and 3, for a total of 147 instances.
The \emph{medium} benchmark set is composed of the instances belonging to sets 5 and 6, for a total of 120 instances.
The \emph{large} benchmark set is composed of the instances belonging to sets 4 and 7, for a total of 120 instances.
A brief summary of the instance sets is reported in Table~\ref{tab:instances}.
%
%
To represent the stochastic travel times between two nodes $i$ and $j$, the random variables $\xi_{ij} \sim \mathcal{N}(\mu_{ij}, \sigma_{ij}^2)$ were defined by setting:
\[\mu_{ij} = t_{ij} \quad \text{and} \quad \sigma_{ij}^2 = \sqrt{t_{ij}},\]
where the quantity $t_{ij}$ is the Euclidean distance between nodes $i$ and $j$.

\subsection{Computational environment}
\label{ssec:environment}

The 2ILM was implemented in C++ using CPLEX 22.1.1 Concert Technology for the callback implementation and run in single-threaded mode with a time limit of 1200 seconds. The code was compiled on Ubuntu 22.04. Experiments were carried out on a 64-bit Linux machine equipped with an AMD EPYC 7532 (Zen 2) processor at 2.40 GHz and 24 GB of RAM. CPLEX built-in cuts were used in all experiments. 

\subsection{The impact of the upper-bounding functionals}
\label{ubfc}

This section provides an assessment of the impact of the UBFs introduced in Section~\ref{ssec:ubf}. Table~\ref{tab:ubf} presents the results of 2ILM without UBFs separation (left) and with UBFs separation (right). 
The results are further subdivided by instance sizes.

\begin{center}
\begin{tabular}{lrrrrrrrrr}
\toprule
  & & \multicolumn{4}{c}{\textbf{2ILM without UBFs}}  & \multicolumn{4}{c}{\textbf{2ILM with UBFs}}  \\ \vspace*{1pt}
  &  \# & \textbf{Opt} & \textbf{Cpu(s)} & \textbf{Nodes} & \textbf{Gap(\%)} & \textbf{Opt} & \textbf{Cpu(s)} & \textbf{Nodes} & \textbf{Gap(\%)} \\ \cmidrule(r){1-2} \cmidrule(lr){3-6} \cmidrule(l){7-10}
\emph{Small}  & 147 & 104  & 51.3 & 72,862 & 11.92  & 110  & 351.6  & 67,515  & 7.48 \\
\emph{Medium} & 120 & 52 & 18.9 & 55,923 & 25.69  & 60 & 650.2  & 57,395  & 11.16  \\
\emph{Large}  & 120 & 31 & 24.9 & 34,343 & 17.03 & 35 & 878.0  & 31,462  & 11.62  \\ \cmidrule(r){1-2} \cmidrule(lr){3-6} \cmidrule(l){7-10}
\textbf{ALL}  & 387 & 187  & 75.9 & 55,666 & 18.88  & 205  & 607.4  & 53,198  & 10.62  \\ \bottomrule
\end{tabular}
\captionof{table}{Computational results without UBFs separation (left) and with UBFs separation (right).}
\label{tab:ubf}
\end{center}

The columns \#, Opt, Cpu(s), Nodes and Gap(\%) report the number of instances, the number of optimal instances solved, the average running time (in seconds) for the instances solved to optimality, the average number of explored branching nodes and the average integrality gap, respectively.
As shown in the table, the inclusion of UBFs significantly improves the performance of the 2ILM. In terms of the number of optimally solved instances, the 2ILM with UBFs consistently outperforms its counterpart, proving the optimality for 18 additional instances. Regarding the average running time to prove optimality, the 2ILM without UBFs is faster because it can only solve the easiest instances. Finally, when considering the average integrality gap, the 2ILM with UBFs achieves substantially better results.

\subsection{The computational results of the 2ILM}
\label{results}

Tables~\ref{tab:small},~\ref{tab:medium} and~\ref{tab:large} report the computational results for the small, medium and large benchmark sets, respectively. Table~\ref{tab:all} gives a summary of the computational results for all the instances. The tables are organised in four parts. The first and the second parts report the selected family of optimality cuts and the constraints liftings adopted, respectively. The third and the fourth parts describe the statistics for the 2ILM method and for the upper bounds, respectively. Columns Opt, Cpu(s), Gap(\%), Nodes and Cuts report the number of instances solved to optimality, the average running time (in seconds) for the instances solved to optimality, the average integrality gap, the average number of explored branching nodes and the average number of cuts added, respectively.
\begin{center}
    \begin{tabular}{@{}llllllrrrrrrrr@{}}
        \toprule
        \multicolumn{3}{c}{\textbf{Family}}  & \multicolumn{3}{c}{\textbf{Lifting}}  & \multicolumn{5}{c}{\textbf{2ILM}}  & \multicolumn{3}{c}{\textbf{Upper Bounds}} \vspace*{1pt}  \\
        \textbf{1} & \textbf{2} & \textbf{3} & \textbf{A} & \textbf{B} & \textbf{C} & \textbf{Opt} & \textbf{Cpu(s)} & \textbf{Gap(\%)} & \textbf{Nodes} & \textbf{Cuts} & \textbf{$\text{P}^m_{\text{UB}}\text{(\%)}$} & \textbf{$\text{C}^m_{\text{UB}}\text{(\%)}$} & \textbf{Cpu(s)} \\ \cmidrule(r){1-3} \cmidrule(lr){4-6} \cmidrule(lr){7-11} \cmidrule(l){12-14}

        \checkmark  &  &  & \checkmark  &  & & 107  & 43.53  & 7.46  & 66,270  & 736  & 7.13  & 22.50  & 0.40  \\
        \checkmark  &  &  &  & \checkmark  & & 102  & 56.83  & 11.75  & 71,623  & 816  & 8.29 & 23.99  & 0.37  \\
        \checkmark  &  &  &  &  & \checkmark  & 86 & 28.88  & 24.27  & 92,265  & 1,345  & 9.26 & 27.79  & 0.38  \\
        & \checkmark  &  & \checkmark  &  &  & 109  & 61.43  & 7.82 & 65,742 & 672 & 7.37 & 22.55  & 0.40  \\
        & \checkmark  &  &  & \checkmark  & & 101  & 48.78  & 11.57  & 68,848  & 802  & 8.09 & 23.90  & 0.34 \\
        & \checkmark  &  &  &  & \checkmark  & 86 & 33.42  & 24.28  & 90,336  & 1,350  & 9.22 & 27.79  & 0.35  \\
        &  & \checkmark  & \checkmark  &  &  & 110 & 65.08 & 7.48 & 67,515  & 643 & 7.39 & 22.46 & 0.39  \\
        &  & \checkmark  &  & \checkmark  &  & 102  & 36.24  & 11.60  & 71,526  & 762 & 8.28 & 23.99  & 0.34 \\
        &  & \checkmark  &  &  & \checkmark  & 87 & 51.87  & 25.01  & 94,006  & 1,286 & 9.37 & 27.75  & 0.37  \\ \bottomrule
    \end{tabular}
	\captionof{table}{Results for the 147 small instances. The families of optimality cuts are: (1) Solution-based, (2) Route-based, (3) Customer-based. The constraints liftings are: (A) Sedrakyan, (B) Time budget, (C) No lifting.}
	\label{tab:small}
\end{center}

Column $\text{P}^m_{\text{UB}}\text{(\%)}$ reports the average relative differences between the profit upper bound and the value of an optimal solution for the instances solved to optimality. Column $\text{C}^m_{\text{UB}}\text{(\%)}$ reports the average relative differences between the customers upper bound and the number of visited customers in an optimal solution for the instances solved to optimality. Finally, the two columns Cpu(s) describe the average running time (in seconds) to prove optimality (for the 2ILM method) and the average running time (in seconds) for computing the upper bounds, respectively.

\begin{center}
    \begin{tabular}{llllllrrrrrrrr}
        \toprule
        \multicolumn{3}{c}{\textbf{Family}}  & \multicolumn{3}{c}{\textbf{Lifting}}  & \multicolumn{5}{c}{\textbf{2ILM}}  & \multicolumn{3}{c}{\textbf{Upper Bounds}} \vspace*{1pt}  \\
        \textbf{1} & \textbf{2} & \textbf{3} & \textbf{A} & \textbf{B} & \textbf{C} & \textbf{Opt} & \textbf{Cpu(s)} & \textbf{Gap(\%)} & \textbf{Nodes} & \textbf{Cuts} & \textbf{$\text{P}^m_{\text{UB}}\text{(\%)}$} & \textbf{$\text{C}^m_{\text{UB}}\text{(\%)}$} & \textbf{Cpu(s)} \\ \cmidrule(r){1-3} \cmidrule(lr){4-6} \cmidrule(lr){7-11} \cmidrule(l){12-14}

        \checkmark  &  &  & \checkmark  &  & & 57 & 38.88 & 10.96 & 55,956  & 1,506  & 12.83 & 21.05  & 7.08 \\
        \checkmark  &  &  &  & \checkmark &  & 52 & 17.81 & 25.94  & 52,645 & 1,588 & 16.44  & 25.07  & 7.27 \\
        \checkmark  &  &  &  &  & \checkmark  & 50 & 30.74 & 55.82  & 56,769  & 2,682 & 24.42  & 32.90  & 9.01 \\
        & \checkmark  &  & \checkmark  &  &  & 59 & 72.33  & 11.26  & 59,025 & 1,235  & 13.49  & 20.37  & 6.32 \\
        & \checkmark  &  &  & \checkmark  & & 52 & 23.29 & 25.73  & 55,482  & 1,515 & 16.37  & 25.07  & 6.20  \\
        & \checkmark  &  &  &  & \checkmark  & 49 & 18.77 & 55.35  & 57,416  & 2,535  & 22.94  & 32.35  & 8.80 \\
        &  & \checkmark  & \checkmark  &  &  & 60  & 81.00  & 11.16  & 57,395  & 1,211 & 13.37  & 20.09 & 6.79  \\
        &  & \checkmark  &  & \checkmark  &  & 52 & 25.30 & 25.73  & 54,983  & 1,525  & 16.34  & 25.07  & 7.02  \\
        &  & \checkmark  &  &  & \checkmark  & 48 & 1.97 & 54.79  & 55,946  & 2,777 & 22.18  & 31.24  & 8.39 \\ \bottomrule
    \end{tabular}
	\captionof{table}{Results for the 120 medium instances. The families of optimality cuts are: (1) Solution-based, (2) Route-based, (3) Customer-based. The constraints liftings are: (A) Sedrakyan, (B) Time budget, (C) No lifting.}
	\label{tab:medium}
\end{center}

\begin{center}
    \begin{tabular}{llllllrrrrrrrr}
        \toprule
        \multicolumn{3}{c}{\textbf{Family}}  & \multicolumn{3}{c}{\textbf{Lifting}}  & \multicolumn{5}{c}{\textbf{2ILM}}  & \multicolumn{3}{c}{\textbf{Upper Bounds}} \vspace*{1pt}  \\
        \textbf{1} & \textbf{2} & \textbf{3} & \textbf{A} & \textbf{B} & \textbf{C} & \textbf{Opt} & \textbf{Cpu(s)} & \textbf{Gap(\%)} & \textbf{Nodes} & \textbf{Cuts} & \textbf{$\text{P}^m_{\text{UB}}\text{(\%)}$} & \textbf{$\text{C}^m_{\text{UB}}\text{(\%)}$} & \textbf{Cpu(s)} \\ \cmidrule(r){1-3} \cmidrule(lr){4-6} \cmidrule(lr){7-11} \cmidrule(l){12-14}

        \checkmark  &  &  & \checkmark  &  &  & 34 & 74.01 & 11.33 & 37,485  & 989  & 9.08  & 25.12  & 167.95 \\
        \checkmark  &  &  &  & \checkmark  &  & 31 & 20.80 & 16.96  & 34,583  & 1,114 & 10.74  & 26.91  & 137.89 \\
        \checkmark  &  &  &  &  & \checkmark  & 30 & 92.39 & 27.54  & 38,043  & 1,506 & 11.71  & 28.64  & 152.62 \\
        & \checkmark  &  & \checkmark  &  &  & 33 & 49.48 & 11.35  & 34,763  & 1,060  & 9.08  & 24.81 & 162.97 \\
        & \checkmark  &  &  & \checkmark  & & 31 & 21.49 & 17.10  & 33,348  & 1,153 & 10.74  & 26.58  & 119.18 \\
        & \checkmark  &  &  &  & \checkmark  & 30 & 55.36 & 27.41  & 37,601  & 1,512  & 11.76  & 28.34  & 126.02 \\
        &  & \checkmark  & \checkmark  &  &  & 35  & 90.35  & 11.62  & 31,462 & 1,062 & 9.89 & 25.27  & 138.41 \\
        &  & \checkmark  &  & \checkmark  & & 31 & 24.12 & 17.00  & 31,760  & 1,139  & 10.69  & 26.58  & 105.44  \\
        &  & \checkmark  &  &  & \checkmark  & 29 & 30.78 & 27.35  & 35,811  & 1,519 & 11.14  & 28.37  & 111.82 \\ \bottomrule
    \end{tabular}
	\captionof{table}{Results for the 120 large instances. The families of optimality cuts are: (1) Solution-based, (2) Route-based, (3) Customer-based. The constraints liftings are: (A) Sedrakyan, (B) Time budget, (C) No lifting.}
	\label{tab:large}
\end{center}

When customer-based optimality cuts are combined with Sedrakyan lifting, the 2ILM method is able to solve approximately 75\%, 50\%, and 30\% of the small, medium, and large benchmark sets, respectively.

The average CPU running time required to prove optimality is similar across each benchmark set, even though the number of instances solved to optimality differs. This is because the majority of the solved instances are characterised by a lower amount of time budget $T_{\max}$. The average integrality gap appears slightly lower for the small benchmark set, while it tends to increase for the medium and large benchmark sets. The average number of explored nodes and cuts are comparable across all benchmark sets. Regarding the upper bounds section of the table, the average profit upper bound and customers upper bound gaps are also quite similar across all the benchmark sets, for the same reason as discussed above for the average CPU running time. Finally, the average CPU time for computing the upper bounds increases with the size of the instances, as expected.

\begin{center}
    \begin{tabular}{llllllrrrrrrrr}
        \toprule
        \multicolumn{3}{c}{\textbf{Family}}  & \multicolumn{3}{c}{\textbf{Lifting}}  & \multicolumn{5}{c}{\textbf{2ILM}}  & \multicolumn{3}{c}{\textbf{Upper Bounds}} \vspace*{1pt}  \\
        \textbf{1} & \textbf{2} & \textbf{3} & \textbf{A} & \textbf{B} & \textbf{C} & \textbf{Opt} & \textbf{Cpu(s)} & \textbf{Gap(\%)} & \textbf{Nodes} & \textbf{Cuts} & \textbf{$\text{P}^m_{\text{UB}}\text{(\%)}$} & \textbf{$\text{C}^m_{\text{UB}}\text{(\%)}$} & \textbf{Cpu(s)} \\ \cmidrule(r){1-3} \cmidrule(lr){4-6} \cmidrule(lr){7-11} \cmidrule(l){12-14}

        \checkmark  &  &  & \checkmark  &  & & 198  & 47.42 & 10.39 & 54,146  & 1,053  & 8.46  & 22.68  & 54.43  \\
        \checkmark  &  &  &  & \checkmark  & & 185  & 39.83 & 18.82  & 54,253  & 1,148 & 9.85 & 24.59  & 45.15  \\
        \checkmark  &  &  &  &  & \checkmark  & 166  & 40.92 & 35.59  & 64,445  & 1,810  & 11.99  & 28.71  & 50.26  \\
        & \checkmark  &  & \checkmark  &  &  & 201  & 62.67 & 10.60  & 54,053  & 967  & 8.72 & 22.50  & 52.65  \\
        & \checkmark  &  &  & \checkmark  &  & 184  & 36.98 & 18.74  & 53,696  & 1,132  & 9.71 & 24.48  & 39.01  \\
        & \checkmark  &  &  &  & \checkmark  & 165  & 33.06 & 35.49  & 63,776  & 1,768 & 11.64  & 28.54  & 41.94  \\
        &  & \checkmark  & \checkmark  &  &  & 205 & 74.05  & 10.62  & 53,198 & 949 & 8.87 & 22.48 & 45.17  \\
        &  & \checkmark  &  & \checkmark  & & 185  & 31.14 & 18.74  & 54,066  & 1,115 & 9.82 & 24.54  & 35.00 \\
        &  & \checkmark  &  &  & \checkmark  & 164  & 33.54 & 35.58  & 64,159  & 1,821  & 11.41  & 28.33  & 37.42  \\ \bottomrule
    \end{tabular}
	\captionof{table}{Results for all 387 instances. The families of optimality cuts are: (1) Solution-based, (2) Route-based, (3) Customer-based. The constraints liftings are: (A) Sedrakyan, (B) Time budget, (C) No lifting.}
	\label{tab:all}
\end{center}

For the complete set of instances, the algorithm is able to solve about 53\% of them when using the customer-based optimality cuts in combination with the Sedrakyan lifting. The customer-based optimality cuts appears to yield the best results in terms of the number of optimal solutions, while the Sedrakyan and Time-budget liftings clearly outperform the model without liftings on nearly every measure.

\subsection{The impact of parameter \texorpdfstring{$\alpha$}{alpha}}
\label{alpha}

We conclude the computational analysis by examining the impact of parameter $\alpha$ for different values of stochasticity.
Table~\ref{tab:alpha} reports the results obtained on the complete TOP benchmark set by varying the value of $\alpha$. The table follows the same structure as those presented in Section~\ref{results}.

\begin{center}
\begin{tabular}{crrrrrrrr}
\toprule
 & \multicolumn{5}{c}{\textbf{2ILM}}  & \multicolumn{3}{c}{\textbf{Upper Bounds}} \\
$\alpha$ & \textbf{Opt} & \textbf{Cpu(s)} & \textbf{Gap(\%)} & \textbf{Nodes} & \textbf{Cuts} & \multicolumn{1}{c}{\textbf{$\text{P}^m_{\text{UB}}\text{(\%)}$}} & \multicolumn{1}{c}{\textbf{$\text{C}^m_{\text{UB}}\text{(\%)}$}} & \multicolumn{1}{c}{\textbf{Cpu(s)}} \\ \cmidrule(r){1-1} \cmidrule(lr){2-6} \cmidrule(l){7-9}
0.75 & 159  & 48.48 & 24.35  & 68,933 & 1,446.69 & 12.34  & 20.54  & 56.42 \\
0.85 & 173  & 43.64 & 17.13  & 62,147 & 1,197.11 & 10.49  & 21.92  & 41.78 \\
0.90 & 188  & 57.43 & 13.88  & 56,429 & 1,054.87 & 9.66 & 22.35  & 48.95 \\
0.95 & 205  & 74.05 & 10.62  & 53,198 & 949.29  & 8.87 & 22.48  & 45.17 \\
0.99 & 218  & 56.53 & 9.00 & 48,283 & 849.98  & 8.31 & 24.29  & 53.61 \\ \bottomrule
\end{tabular}
\captionof{table}{Computational results for all instances by varying parameter $\alpha$.}
\label{tab:alpha}
\end{center}

As shown in Table~\ref{tab:alpha}, increasing $\alpha$ appears to simplify the problem. Specifically, higher values of $\alpha$ lead to a larger number of instances being solved to optimality and to improved solution quality, as reflected by lower average gaps. Furthermore, the trends in the average number of explored nodes and added cuts indicate that larger values of $\alpha$ generally result in fewer explored nodes and a reduced number of cuts. Regarding the average solution time, higher values of $\alpha$ generally lead to longer times to certify optimality, with the exception of $\alpha = 0.99$. Finally, when examining the upper-bound part of the table, larger values of $\alpha$ yield tighter bounds on the expected profit but looser bounds on the average number of served customers.

\section{Conclusions}
\label{sec:end}

This paper studied the Stochastic Team Orienteering Problem (STOP), an extension of the Team Orienteering Problem (TOP) that addresses uncertainty in travel times. We modelled the STOP as a two-stage stochastic integer programming problem and formulated a two-index mathematical model for the first-stage problem.
We implemented an exact method based on the Integer L-shaped method to solve the STOP. Such a method adopts a new set of optimality cuts aimed to linearise the objective function which is expressed as the expected profit collected by the selected routes. We introduced a set of valid inequalities and constraints liftings to tighten the mathematical formulation of the first-stage problem. We conducted a comparative analysis to assess the impact of the Upper-Bounding Functionals (UBFs) showing that their inclusion significantly affects the performance of the algorithm. Computational results on the TOP benchmark dataset are presented to evaluate the effectiveness of the two-index formulation. Finally, the study examines the algorithm’s performance under different levels of stochasticity.



\bibliographystyle{cas-model2-names}

\bibliography{cas-refs}

\end{document}